\documentclass{amsart}

\usepackage{a4wide, mathtools}
\usepackage{cmap}
\usepackage[utf8]{inputenc}
\usepackage[T1]{fontenc}
\usepackage{graphicx,newtxtext,url,times,float, color}

\usepackage[none]{hyphenat}

\usepackage[center]{caption}

\usepackage{tabularx,amsmath,amssymb,amsthm,amsfonts,bm,cite,hhline, tikz-cd, pgfplots, pgfplotstable, subcaption}

\newtheorem{theorem}{Theorem}[section]
\newtheorem{lemma}[theorem]{Lemma}

\newtheorem{conjecture}[theorem]{Conjecture}

\newtheorem{definition}[theorem]{Definition}

\numberwithin{equation}{section}

\numberwithin{equation}{section}

\usepackage[none]{hyphenat}

\title[Note on Collatz conjecture]{Note on Collatz conjecture}
\date{}
\author{Abdelrahman Ramzy}
\address{Department of Mathematics, Faculty of Education, Al-Azhar University, Cairo, Egypt}
\email{hosam7101996@gmail.com}
\begin{document}

\begin{abstract}
In this paper, we show that if the numbers in the range $[1,2^n]$ satisfy Collatz conjecture, then almost all integers in the range $[2^n+1,2^{n+1}]$ will satisfy the conjecture as $n \to \infty$. The previous statement is equivalent to claiming that almost all integers in $[2^n+1,2^{n+1}]$ will iterate to a number less than $2^n$. This actually has been proved by many previous results. But in this paper we prove this claim using different methods. We also utilize our assumption on numbers in $[1,2^n]$ to show that there are a set of integers (denoted by $p\operatorname{-}C$) whose seem to be not iterating to a number less than $2^n$, but since they are connected to Collatz numbers in $[1,2^n]$ they eventually will iterate to $1$. We address the distribution of $p\operatorname{-}C$ and give an explicit formula which computes a lower bound to the number of these integers. We also show (computationally) that the number of $p\operatorname{-}C$ in a given interval is proportional to the numbers of another set of incidental Collatz numbers in the same interval (whose distribution is completely unpredictable).
\end{abstract}
\maketitle

\section{Introduction}
The well-known Collatz conjecture or $3k+1$ conjecture depends on performing the following two operations on a random positive integer:

\begin{itemize}
	\item If the number is odd, multiply it by $3$ and add $1$.
	\item If the number is even, divide it by $2$.
\end{itemize}

Now, repeat these operations and take the result of each operation as the input, Collatz conjecture states that this chain of operations will eventually reach $1$. For example, if we begin with the number $3$, we will get $3 \Longrightarrow 10 \Longrightarrow 5 \Longrightarrow 16 \Longrightarrow 8 \Longrightarrow 4 \Longrightarrow 2 \Longrightarrow 1$. There have been many papers studying this conjecture, but fortunately, the conjecture still not solved (namely, proved or disproved). For example, the authors in \cite{kra} showed that there are at least $x^{0.84}$ integers below $x$ whose are Collatz numbers. While in \cite{ars} the author showed that from a given range $[1, G]$ of Collatz numbers, we can construct finitely many odd Collatz numbers of different forms. Also the author in \cite{zar} addressed Collatz conjecture and simplified it by constructing a streamlined version of the conjecture that involves only integers $K$ such that $K \equiv 2 \bmod 3$. Almost any result about this conjecture involves the word \textbf{almost}, which appears in our paper and in \cite{Tao} as well, but almost all of these results still very great. For example, the great results in \cite{Tao} show that almost all Collatz orbits attain almost bounded values, which means that finding a counter-example to the conjecture is very likely impossible. 

In this paper, we show that if the integers in $[1,2^n]$ are Collatz numbers, then almost all integers in $[2^n+1,2^{n+1}]$ are collatz numbers as well (as $n \to \infty$). This statement means that almost all integers in $[2^n+1,2^{n+1}]$ will iterate to a number less that $2^n$, which has been proved by many previous results (see \cite{Kor} and \cite{ALL} for instance). But in this paper we prove this statement using different method. We also address a set of integers (denoted by $p\operatorname{-}C$) whose not iterating to a number less than $2^n$ (in the short-term iterations), and show that they are eventually will iterate to a number less than $2^n$ and consequentially to $1$. We also define incidental Collatz integers and show (computationally) that the ratios of $p\operatorname{-}C$ to the number of incidental Collatz integers are almost equal for any value of $n$. 

The paper is organized as follows. In section 2 we define the chain of an integer, and then we address odd integers in $[2^n+1,2^{n+1}]$ and show that each one of them has a unique chain. In section 3 we give a formula which computes a lower bound of Collatz integers in the range $[2^n+1,2^{n+1}]$. In Section 4 we show the main idea behind $p\operatorname{-}C$ and give some definitions in order to help us construct our formula. In Section 5 we give a formula which computes a lower bound to th number of $p\operatorname{-}C$ and give some computations which support our claim about the proportion between $p\operatorname{-}C$ and incidental Collatz integers.

\section{Proving the uniqueness of the chains in the range $[2^n+1,2^{n+1}]$}

Assume that the numbers in $[1,2^n]$ are Collatz numbers (namely, satisfying Collatz conjecture), if $L \in [2^n+1,2^{n+1}]$ and $L$ is even, then $L$ will be sent directly to the range $[1,2^n]$. Hence, there are $2^{n-1}$ integers in $[2^n+1,2^{n+1}]$ satisfying Collatz conjecture (namely, even numbers). Now, we will study odd numbers in $[2^n+1,2^{n+1}]$.  If $L$ is odd, then it is obvious that we will start the chain by the function $A(L)=3L+1$ followed by $B(A(L))=(3L+1)/2$, and then we won't know if $BA(L)$ should composed by $A$ or $B$ since we don't know if $BA(L)$ is odd or even. Now, let $C_n(L)$ be the chain of $L$ which starts with $BA$, ends with $B$ and contains $n$ $B$ functions, namely, $C_n(L)=B..........BA(L)$ and $\beta (C_n)=n$ regardless of the value of $\alpha (C_n)$ (such that $\beta (C_n)$ is the number of $B$ functions in $C_n(L)$ and $\alpha (C_n)$ is the number of $A$ functions in $C_n(L)$). In other words, the chain of $L$ is divided to a number of cells which equals to $\beta (C_n)=n$, and each cell either contains $B$ or $BA$ . We clarified why $C_n(L)$ should start with $BA$, but why should it end with $B$? it is because the main purpose of the chain is to send $L$ to the range $[1,2^n]$, therefore, it should end with a decrement, namely the function $B$. 

Now, it is obvious that  $1 \leq \alpha (C_n) \leq n$, but since $C_n(L)$ starts with $BA$, then the number of different $C_n$ chains equals to $\binom{n-1}{0} + \binom{n-1}{1} + \binom{n-1}{2} ..... + \binom{n-1}{n-1} = 2^{n-1}$ which is equivalent to the number of odd integers in $[2^n+1,2^{n+1}]$. To study these chains and decide which one sends it's odd integer to the range $[1,2^{n-1}]$, we should prove that every odd integer $L \in [2^n+1,2^{n+1}]$ has a unique chain $C_n$. To prove that we require the following theorem. Note that the following theorem is equivalent to the periodicity theorem in \cite{Ter}, but we rephrase it here in a different way in order to help us prove Theorem \ref{Thm12} which will be used later to construct the main formulas in Sections $3$ and $5$.


\begin{theorem} \label{Thm11}
	If $M=2^z+L$, then the order of the functions in $C_z(M)$ will be identical to the order in $C_z(L)$, and if the cell after $C_z(L)$ contains $B$, then the cell after $C_z(M)$ will contain $BA$ and vice versa.
\end{theorem} 

\begin{proof}
	If $M=2^z+L$ and $L$ is odd, then $M$ is odd and $A(M)=3(2^z+L)+1=3*2^z+A(L)$. Also if $L$ is even, then $M$ is even and $B(M)=\frac{2^z+L}{2}=2^{z-1}+B(L)$. Therefore, we guarantee that the chain of $M$ will be identical to the chain of $L$ as long as the chain of $M$ equals to the chain of $L$ plus even number,  which depends on the value of $z$. Hence, $C_z(M)$ is identical to $C_z(L)$, $C_z(M)=\frac{3^{\alpha(C_z(M))}*2^z}{2^z}+C_z(L)=3^{\alpha(C_z(M))}+C_z(L)$ and if the cell after $C_z(L)$ contains $B$ (namely, $C_z(L)$ is even), then the cell after $C_z(M)$ will contain $BA$ and vice versa.
\end{proof}	

\begin{theorem}\label{Thm12}
	Every odd integer in $[2^n+1,2^{n}]$ has a unique chain $C_n$.
\end{theorem} 

\begin{proof}
	Assume that $L, M \in [2^n+1,2^{n+1}]$, $C_n(L)$ is identical to $C_n(M)$, $L < M$ and $M-L=2^{a_1}+ 2^{a_2}+......+2^{a_r}$ such that $ 1 \le a_1 < a_2 < ...... <a_r$, then $C_{a_1+1}(M)$ is identical to $C_{a_1+1}(L)$. But $C_{a_1+1}(M)$ is not identical to $C_{a_1+1}(L+2^{a_1})$ by Theorem \ref{Thm11}. Now, we have $C_{a_1+1}(L+2^{a_1})$ is identical to $C_{a_1+1}(L+2^{a_1}+2^{a_2})$ since $a_2 \geq a_1+1$, $C_{a_1+1}(L+2^{a_1}+2^{a_2})$ is identical to $C_{a_1+1}(L+2^{a_1}+2^{a_2}+2^{a_3})$ since $a_3 > a_1+1$, ....., $C_{a_1+1}(L+2^{a_1}+2^{a_2}+ ...... +2^{a_{r-1}})$ is identical to $C_{a_1+1}(L+2^{a_1}+2^{a_2}+2^{a_3}+ ...... +2^{a_r})=C_{a_1+1}(M)$ since $a_r > a_1+1$. Consequently, $C_{a_1+1}(L+2^{a_1})$ is identical to $C_{a_1+1}(M)$ which is clearly a contradiction and our lemme is proved.
\end{proof}	

\section{A lower bound to the number of Collatz integers in the range $[2^n+1,2^{n+1}]$}

Since we proved that each odd integer in $[2^n+1,2^{n+1}]$ has a unique chain $C_n$, we can give a formula which computes the number of chains that sends it's own integer to the range $[1,2^n]$. But first, we should acknowledge that the formula depends on the assumption that the addition of $1$ in the function $A$ can be ignored, namely the function $A$ can be reduced to multiplying by $3$. In other words, the formula depends on the following conjecture which seems very likely to be true.
\begin{conjecture} \label{con1}
	Assume that odd integers in $[2^n+1,2^{n}]$ are very close to $2^{n+1}$. If $2^{\beta(C_n)} > 2* 3^{\alpha(C_n)} \Longrightarrow 2^{\beta(C_n)-1} > 3^{\alpha(C_n)} \Longrightarrow \frac{2^{\beta(C_n)-1}}{3^{\alpha(C_n)}} > 1$, then either $C_n(L)$ in $[1, 2^n]$, or it starts with a chain that is in $[1, 2^n]$. In other words, $C_n(L)-\frac{3^{\alpha(C_n)}*L}{2^{\beta(C_n)}} \leq 2^n-\frac{3^{\alpha(C_n)}*L}{2^{\beta(C_n)}}$ (or it starts with such a chain).
\end{conjecture} 

Now, from the following inequality:

\begin{equation}\label{eq}
	\frac{2^{\beta(C_n)-1}}{3^{\alpha(C_n)}} > 1
\end{equation}

it is obvious that any chain with $\alpha < (\beta-1)\log_3(2) \Longrightarrow \alpha < (n-1)\log_3(2)$ will satisfy Collatz conjecture. Therefore, we will have a number of satisfying chains (those chains are official Collatz chains according to Definition \ref{def1}) that equals to:

\begin{equation}\label{eq1}
	\sum_{x=1}^{\lfloor (n-1)\log_3(2) \rfloor} \binom{n-1}{x-1} .
\end{equation}

Note that in the remained chains, we will have $\alpha > (n-1)\log_3(2)$ which should mean that they are not Collatz chains. But this is wrong, since in many cases we could have $\alpha > (n-1)\log_3(2)$ in total while $C_n(L)$ have sent $L$ to the range $[1,2^n]$ at some point (namely, a chain might not satisfy (\ref{eq}), but it may start with satisfying one), for example:

$$C_5(61)=BABABBBA.$$

In order to compute the exact number of these chains, we require the following definitions:

\begin{definition}\label{def1}
	Official-Collatz chain is a chain that satisfies (\ref{eq}).
\end{definition} 

\begin{definition}\label{def2}
	Non-official-Collatz chain is a chain that doesn't satisfy (\ref{eq}), but starts with satisfying one.
\end{definition}

\begin{definition}\label{def3}
	$u$-comprehensive-Collatz chain is a chain that satisfies (\ref{eq}) with the largest possible $\beta$ such that $\beta - \alpha = u$.
\end{definition}

Official-Collatz chain, Non-official-Collatz chain and $u$-comprehensive-Collatz chain will be denoted by $O\operatorname{-}C$, $N\operatorname{-}o\operatorname{-}C$ and $u\operatorname{-}c\operatorname{-}C$ respectively.

From Definitions \ref{def2} and \ref{def3}, it is obvious that for any $u\operatorname{-}c\operatorname{-}C$ we will have $\frac{3}{2} > \frac{2^{\beta(C_n)-1}}{3^{\alpha(C_n)}} > 1$, $N\operatorname{-}o\operatorname{-}C$ starts with $u\operatorname{-}c\operatorname{-}C$, $u \geq 2$ and if a set of chains are $u\operatorname{-}c\operatorname{-}C$ we can compute $\alpha$ and $\beta$ using the parameter $u$. For example, if a chain is $3\operatorname{-}c\operatorname{-}C$ then $\beta = \lfloor \frac{3*\log_2(3)-1}{\log_2(3)-1} \rfloor =6$ and $\alpha = \lfloor \frac{3-1}{\log_2(3)-1} \rfloor =3$. Therefore, we give the following lemma.

\begin{lemma}
	If a chain is $u\operatorname{-}c\operatorname{-}C$, then $\beta = \lfloor \frac{u*\log_2(3)-1}{\log_2(3)-1} \rfloor$ and $\alpha = \lfloor \frac{u-1}{\log_2(3)-1} \rfloor$.
\end{lemma}

Now, it is obvious that for any $u \geq 2$, there are $\binom{\beta-1}{\alpha-1}=\binom{\lfloor \frac{u*\log_2(3)-1}{\log_2(3)-1} \rfloor-1}{\lfloor \frac{u-1}{\log_2(3)-1} \rfloor-1}$ $u\operatorname{-}c\operatorname{-}C$, and if $\alpha > (n-1)\log_3(2)$ we can construct many chains which starts with $\beta-\alpha\operatorname{-}c\operatorname{-}C$, $\beta-\alpha-1\operatorname{-}c\operatorname{-}C$, ..... or $2\operatorname{-}c\operatorname{-}C$. Therefore, each set of chains with $n-2 \geq \alpha > (n-1)\log_3(2)$ contains a number of $N\operatorname{-}o\operatorname{-}C$ more than or equals to 

$$\sum_{x=\alpha}^{n-2} \binom{\lfloor \frac{(n-x)*\log_2(3)-1}{\log_2(3)-1} \rfloor-1}{\lfloor \frac{n-x-1}{\log_2(3)-1} \rfloor-1}.$$

Consequently, there is a number of $N\operatorname{-}o\operatorname{-}C$ more than or equal to:

\begin{equation}\label{eq2}
	\sum_{y=\lceil (n-1)\log_3(2) \rceil}^{n-2} \sum_{x=y}^{n-2} \binom{\lfloor \frac{(n-x)*\log_2(3)-1}{\log_2(3)-1} \rfloor-1}{\lfloor \frac{n-x-1}{\log_2(3)-1} \rfloor-1}.
\end{equation}

But, in order to compute the exact value of $N\operatorname{-}o\operatorname{-}C$, we should consider the combinations of the second part of the chain. for example, if $\beta = 11$ and $\alpha  = 7$ for a set of chains, we can construct many $N\operatorname{-}o\operatorname{-}C$ which starts with $3\operatorname{-}c\operatorname{-}C$ such as $\textcolor{blue}{B}BABABABA\textcolor{red}{BBABBABBA}$. We also can move the free function $\textcolor{blue}{B}$ in the previous chain to get $BA\textcolor{blue}{B}BABABA\textcolor{red}{BBABBABBA}$. Note that if we moved the free function again we will get the chain $BABA\textcolor{blue}{B}BABA\textcolor{red}{BBABBABBA}$. But we counted that chain in the set of $N\operatorname{-}o\operatorname{-}C$ which starts with $4\operatorname{-}c\operatorname{-}C$. Therefore, every free function $B$ must have a limited number of movements in order to prevent any intersections. In order to compute the allowed number of movements for each free function, we must compute the length of the wall $W=....BABABA$ which allows us to put a free function behind it without having an intersection with another $v\operatorname{-}c\operatorname{-}C$ such that $v>u$. For example, the length of the wall which allows us to put $B$ function behind it in the chain $BBABABABA\textcolor{red}{BBABBABBA}$ is three, which can be easily computed by the formula $\lceil \frac{\textcolor{green}{1}+\beta(3\operatorname{-}c\operatorname{-}C)-1-\alpha(3\operatorname{-}c\operatorname{-}C)*\log_2(3)}{\log_2(3)-1} \rceil=3$, and then the number of allowed movements equals to

$$w(1)=\alpha+1-\alpha(3\operatorname{-}c\operatorname{-}C)-\lceil \frac{\textcolor{green}{1}+\beta(3\operatorname{-}c\operatorname{-}C)-1-\alpha(3\operatorname{-}c\operatorname{-}C)*\log_2(3)}{\log_2(3)-1} \rceil=2.$$

Note that the number of allowed movements for the second free function (if it existed) is $w(2)$ (and so on), and the sentence "allowed number of movements" doesn't mean random movements, it means some movements that starts from the end of the total chain and ends at the end of the wall, and the movement doesn't count if it doesn't lead to a new combination. For example, if a chain ends with two free functions $\textcolor{blue}{B_2B_1}.....$ and the second function allowed to move number of movements, and we moved it to get $\textcolor{blue}{B_1B_2.....}$, then this movement doesn't count, since $\textcolor{blue}{B_2B_1}.....$ is identical to $\textcolor{blue}{B_1B_2.....}$. Therefore, the total number of combinations of the second part of the chain doesn't equal to $\displaystyle \prod_{s=1}^{r} w(s)$, since the free functions are identical (they are just $B$ functions). Hence, to compute the exact number of combinations, we recall the following case "Assume that the sets $A_1 \subset A_2
\subset A_3 \subset ..... \subset A_n$ have
the orders $k_1, k_2, k_3, ... , k_n$
respectively,
then how many unique multisets $\{a_1, a_2,
a_3, ...... , a_n\}$ (such that $a_1 \in A_1, a_2
\in A_2, ... , a_n \in A_n$) can be obtained
from the above sets ? this can be easily computed by the formula:
$$ \sum_{m_1=1}^{k_1}
\sum_{m_2=m_1}^{k_2} ......
\sum_{m_n=m_{n-1}}^{k_n} (1).$$ 
and this case is identical to our case in this paper". Hence, the number of combinations of the second part equals to:

\begin{equation} \label{eq3}
	\sum_{r_1=1}^{w(s)}
	\sum_{r_2=r_1}^{w(s-1)} ......
	\sum_{r_n=r_{n-1}}^{w(1)} (1).
\end{equation}

Now, in order to get the final formula which computes a lower bound to the number of satisfying Collatz chains, we should notice that some $N\operatorname{-}o\operatorname{-}C$ have no free functions. Therefore, we should disjion them from (\ref{eq2}), also we should add (\ref{eq1}) to (\ref{eq2}) and join (\ref{eq3}) to (\ref{eq2}) after redefining the function $w$ according to the variables in (\ref{eq2}). Hence, the final formula will take the form:

\begin{equation} \label{eq4}
	\begin{gathered}
		\gamma(n):=\sum_{x=1}^{\lfloor (n-1)\log_3(2) \rfloor} \binom{n-1}{x-1}+\sum_{y=\lceil (n-1)\log_3(2) \rceil}^{n-2} \binom{\lfloor \frac{(n-y)*\log_2(3)-1}{\log_2(3)-1} \rfloor-1}{\lfloor \frac{n-y-1}{\log_2(3)-1} \rfloor-1}+	\\
		\sum_{y=\lceil (n-1)\log_3(2) \rceil}^{n-3} \sum_{x=y+1}^{n-2} \binom{\lfloor \frac{(n-x)*\log_2(3)-1}{\log_2(3)-1} \rfloor-1}{\lfloor \frac{n-x-1}{\log_2(3)-1} \rfloor-1} * 
		\sum_{r_1=1}^{w(x-y))}
		\sum_{r_2=r_1}^{w(x-y-1)} ....
		\sum_{r_e=r_{e-1}}^{w(1)} (1)
	\end{gathered}
\end{equation}

such that 

$$w(s)=y+1-\lfloor \frac{n-x-1}{\log_2(3)-1} \rfloor-\lceil \frac{\textcolor{green}{s}+\lfloor \frac{(n-x)*\log_2(3)-1}{\log_2(3)-1} \rfloor-1-\lfloor \frac{n-x-1}{\log_2(3)-1} \rfloor*\log_2(3)}{\log_2(3)-1} \rceil .$$

Although (\ref{eq4}) computes the exact number of official and non official Collatz chains in the range $[2^n+1,2^{n+1}]$, we claim that it gives a lower bound to the number of chains which sends it's own integer to the range $[1,2^n]$. Because of the simple fact that not all odd integers in $[2^n+1,2^{n+1}]$ are very close to $2^{n+1}$. Therefore, there are some incidental chains whose might not be official or non official, but still send their own integers to the range $[1,2^n]$ because of the location of the input integer. For example, if we assume that integers in $[1,2^3]$ are Collatz numbers, then the chain of $9$ will take the form $$C_3(9)=BABBA.$$

It is obvious that the above chain is not official or non official Collatz chain, but it still sends the integer $9$ to the range $[1,2^3]$, and that because of the location of $9$ (since it is not very close to $2^4$). Therefore, our claim about (\ref{eq4}) is true and it gives a lower bound to the number of Collatz chains. 

\begin{definition}
	Incidental Collatz number is a number $L$ which iterates to a number less than $2^n$ although $C_n(L)$ not satisfying \ref{eq}.
\end{definition}

The number of incidental integers or chains in $[2^n+1,2^{n+1}]$ will be denoted by $T(n)$, but such chains can not be computed by a given formula, because even if we have the form of the chain we can not determine it's location in the interval $[2^n+1,2^{n+1}]$. However, in Section 5 we show (computationally) that $T(n)$ is proportional to the set of $p\operatorname{-}C$. In Table \ref{table1} we assume that integers in $[1,2^n]$ are Collatz numbers for many values $n$, and then we use $\gamma(n)$ to compute the number of official and non official chains in $[2^n+1, 2^{n+1}]$. We also consider the number of $T(n)$ in our computations. After that we computed the ratio to show that it quickly converges to $1$, note that we considered even Collatz numbers in $[2^n+1, 2^{n+1}]$ while computing the ratio. Finally, we can easily show that $\frac{\gamma(n+1)+2^{n}}{2^{n+1}} \geq \frac{\gamma(n)+2^{n-1}}{2^n}$. 

\begin{lemma}
	$\frac{\gamma(n+1)+2^{n}}{2^{n+1}} \geq \frac{\gamma(n)+2^{n-1}}{2^n}$.
\end{lemma}

\begin{proof}
	If $C_n$ is official or non official chain, then it will give two official or non official chains in the new range, namely $BAC_n$ and $BC_n$. Therefore, we will always have 
	
	$$\frac{\gamma(n+1)+2^{n}}{2^{n+1}} \geq \frac{\gamma(n)+2^{n-1}}{2^n}.$$
\end{proof}

\begin{table}[h]
	\centering 
	
	\begin{tabularx}{\linewidth}{|c|X|X|X|X|}
		\hline
		$n$  & $\gamma(n)$ & $\gamma (n) + T(n)$ & $\frac{\gamma(n)+2^{n-1}}{2^n}$ & $\frac{\gamma (n) +T(n)+2^{n-1}}{2^n}$ \\
		\hline
		3 & 1 & 1 & 0.625000000000 & 0.625000000000\\ 
		\hline 
		4 & 2 & 3 & 0.625000000000 & 0.687500000000 \\
		\hline
		5 & 6 & 9 & 0.687500000000 & 0.781250000000 \\
		\hline
		6 & 17 & 17 & 0.765625000000 & 0.765625000000 \\
		\hline
		7 & 34 & 40 & 0.765625000000 & 0.812500000000 \\
		\hline 
		8 & 77 & 85 & 0.800781250000 & 0.832031250000 \\  
		\hline
		9 & 177 & 178 & 0.845703125000 & 0.847656250000 \\ 
		\hline
		10 & 354 & 385 & 0.845703125000 & 0.875976562500 \\
		\hline
		11 & 751 & 792 & 0.866699218750 & 0.886718750000 \\
		\hline
		12 & 1502 & 1624 & 0.866699218750 & 0.896484375000 \\
		\hline
		13 & 3117 & 3372 & 0.880493164062 & 0.911621093750 \\
		\hline
		14 & 6565 & 6822 & 0.900695800781 & 0.916381835937 \\
		\hline
		15 & 13130 & 13946 & 0.900695800781 & 0.925598144531 \\
		\hline
		16 & 26958 & 28370 & 0.911346435546 & 0.932891845703 \\
		\hline
		17 & 55882 & 57256 & 0.926345825195 & 0.936828613281 \\ 
		\hline
		18 & 111764 & 116579 & 0.926345825195 & 0.944713592529 \\ 
		\hline
		19 & 227600 & 234910 & 0.934112548828 & 0.948055267333 \\ 
		\hline
		20 & 455200 & 473325 & 0.934112548828 & 0.951397895812 \\
		\hline
		21 & 921833 & 959987 & 0.939564228057 & 0.957757472991 \\
		\hline
		22 & 1878800 & 1926862 & 0.947940826416 & 0.959399700164 \\
		\hline
		23 & 3757600 & 3880688 & 0.947940826416 & 0.962614059448 \\
		\hline
		24 & 7593367 & 7818474 & 0.952599942684 & 0.966017365455 \\
		\hline
		25 & 15415312 & 15687824 & 0.959412097930 & 0.967533588409 \\
		\hline
	\end{tabularx}
	\caption{Some counts of $\gamma(n)$ and $T(n)$}     
	\label{table1}
\end{table}

In the following section we utilize our assumption on numbers in $[1,2^n]$ to show that there are many chains in $[2^n+1,2^{n}]$ which don't satisfy \ref{eq}, but still satisfy Collatz conjecture. Also they don't require Conjecture \ref{con1}, therefore we call them "proper Collatz chains" and denote them by $p-C$. 

\section{The main idea behind $p\operatorname{-}C$}

The chains computed by $\gamma(n)$ satisfy \ref{eq} (namely satisfy \ref{eq} or start with a satisfying chain), while $p-C$ don't. Therefore, there is no intersection between the chains computed by $\gamma(n)$ and $p-C$. But before proceeding to the main idea, we should notice that that the sentence "a chain is in $[a,b]$" means that the initiative value $L \in [a,b]$.

Now, we need to recall that each number $L \in [1,2^n]$ has a unique chain $C_n$. Also each number in $[2^n+1,2^{n+1}]$ has a unique chain $C_n$ and it is identical to it's corresponding chain in $[1,2^n]$. In other words, for each number $L \in [1,2^n]$, $C_n(L)$ is identical to $C_n(L+2^n)$ as we clarified before. If a chain is in $[2^n+1,2^{n+1}]$ and don't satisfy \ref{eq}, we deduce that it is not Collatz chain and then it will not be counted by $\gamma(n)$. But it's identical corresponding chain in $[1,2^n]$ is Collatz chain by the assumption on integers in $[1,2^n]$. Therefore, each number in the chain will satisfy the conjecture (not only the initiative value $L$). Hence, we can use the chains in $[\frac{2}{3}*2^n,2^n]$ which don't satisfy \ref{eq} to generate many $p\operatorname{-}C$. Since if the chain in $[\frac{2}{3}*2^n,2^n]$ and don't satisfy \ref{eq}, then $BA(L) \in  [2^n+1,2^{n+1}]$ and the chain of $BA(L)$ will not satisfy \ref{eq} mostly (which means that we didn't count such chains although they satisfy the conjecture). 
\\
Let's give an example to illustrate the idea. Assume that integers in $[1,2^5]$ are Collatz numbers, and let $L=31$, then 

$$C_5(31)=BABABABABA(31)$$
\\
and we guarantee that $BA(31) \in [2^5+1,2^6]$ (since $31 \in [\frac{2}{3}*2^5,2^5]$). Therefore, the chain 

$$C_5(BA(31))=C_5(47)=\textcolor{red}{B}BABABABA(47)$$
\\
will satisfy the conjecture although it doesn't satisfy \ref{eq}. Note that in general the last cell in $C_n(BA(L))$ will be unknown to us, but  we will consider the worst case and assume that it always contains $B$ function (not $BA$). 

Now, let's define the chains which will be used to generate $p\operatorname{-}C$, and in the next section we will provide a formula which computes a lower bound to the number of generative chains in $[\frac{2}{3}*2^n,2^n]$.   

\begin{definition}
	A chain is called "generative Collatz chain ($g\operatorname{-}C$)" if it is in $[\frac{2}{3}*2^n,2^n]$ and don't satisfy \ref{eq}. 
\end{definition}

But in order to guarantee that the initiative value of $p\operatorname{-}C$ is odd number, we have to make sure that $g\operatorname{-}C$ starts with $BABA(L)$. Since if $g\operatorname{-}C=B...BBA(L)$, then $BA(L)$ must be even, and even numbers in  $[2^n+1,2^{n+1}]$ are the trivial Collatz numbers.

\begin{lemma}\label{lem1}
	$g\operatorname{-}C$ must starts with $C_{n-1}BA(L)$ in order to guarantee that the initiative value of $p\operatorname{-}C$ is odd number.
\end{lemma}

\section{A lower bound to the number of $g\operatorname{-}C$}

In order to get our formula, we should recall that all $g\operatorname{-}C$ in $[\frac{2}{3}*2^n,2^n]$. Therefore, if we split the interval $[1,2^n]$ into four intervals (namely $[1,2^{n-2}]$, $]2^{n-2},2^{n-1}]$, $]2^{n-1},3*2^{n-2}]$ and $]3*2^{n-2}, 2^n]$), we will find that the only concerned chains are the chains in the interval $]3*2^{n-2}, 2^n]$ and some of the chains in $]2^{n-1},3*2^{n-2}]$. Hence, in order to get a clear description to the concerned chains, we have to notice that each chain in $[1,2^n]$  is of the form 

$$\textcolor{blue}{BABA}C_{n-2}(L), \textcolor{blue}{BAB}C_{n-2}(L), \textcolor{blue}{BBA}C_{n-2}(L) \operatorname{ or } \textcolor{blue}{BB}C_{n-2}(L).$$
Therefore, and according Lemma \ref{lem1}, if a chain is $g\operatorname{-}C \in ]3*2^{n-2}, 2^n]$, it would be of the form 

$$\textcolor{blue}{BABA}C_{n-3}BA(L), \textcolor{blue}{BAB}C_{n-3}BA(L), \textcolor{blue}{BBA}C_{n-3}BA(L) \operatorname{ or } \textcolor{blue}{BB}C_{n-3}BA(L).$$
Thus, after assuming that the last cell in $p\operatorname{-}C$ contains $B$ function (which is the worst case), we find that each $p\operatorname{-}C$ generated from $g\operatorname{-}C \in ]3*2^{n-2}, 2^n]$ will be of the form

\begin{equation*}
	\begin{tikzcd}[row sep=huge]
		\textcolor{red}{B}\textcolor{blue}{BABA}C_{n-3}(BA(L)), \textcolor{red}{B}\textcolor{blue}{BAB}C_{n-3}(BA(L)), \textcolor{red}{B}\textcolor{blue}{BBA}C_{n-3}(BA(L)) \operatorname{ or } \textcolor{red}{B}\textcolor{blue}{BB}C_{n-3}(BA(L)). 	
		\arrow[d,swap,"BA(L)=L^{\prime}"] &   			 
		\\ 
		\textcolor{red}{B}\textcolor{blue}{BABA}C_{n-3}(L^{\prime}), \textcolor{red}{B}\textcolor{blue}{BAB}C_{n-3}(L^{\prime}), \textcolor{red}{B}\textcolor{blue}{BBA}C_{n-3}(L^{\prime}) \operatorname{ or } \textcolor{red}{B}\textcolor{blue}{BB}C_{n-3}(L^{\prime}).  & 
	\end{tikzcd}
\end{equation*}

\begin{definition}
	$[a,b]$ $p\operatorname{-}C$ are $p\operatorname{-}C$ which are generated from $g\operatorname{-}C$ in the interval $[a,b]$.
\end{definition}

Now, in order to investigate the concerned intervals, we must remember that $g\operatorname{-}C$ are in $[\frac{2}{3}*2^n,2^n]$. Therefore, the concerned intervals can be stated as follows.

\begin{definition}
	The $K\operatorname{-}th$ concerned interval is $I_K=]\frac{2^{2(K+1)-1}+1}{3*2^{2(K)}}*2^n,\frac{2^{2K-1}+1}{3*2^{2(K-1)}}*2^n]$.
\end{definition}

For example; $I_1=]3*2^{n-2}, 2^n]$ and it represents a quarter of the interval $[1,2^n]$, also $I_2=]11*2^{n-4},3*2^{n-2}]$ and it represents a quarter of the interval $]2^{n-1},3*2^{n-2}]$ (which is a quarter of $[1,2^n]$) and so on. Note that all the intervals $I_K$ are subsets of $[\frac{2}{3}*2^n,2^n]$. Now we need to determine the general formula of $I_K$ $p\operatorname{-}C$. We must assume that each $I_1$ $p\operatorname{-}C=\textcolor{red}{B}\textcolor{blue}{BB}C_{n-3}(L^{\prime})$. Since then we will guarantee that such a chain won't satisfy \ref{eq} even if it actually has another formula (namely, $\textcolor{red}{B}\textcolor{blue}{BABA}C_{n-3}(L^{\prime}), \textcolor{red}{B}\textcolor{blue}{BAB}C_{n-3}(L^{\prime}) \operatorname{ or } \textcolor{red}{B}\textcolor{blue}{BBA}C_{n-3}(L^{\prime})$). As for $I_{K \geq 2}$ $p\operatorname{-}C$, we can assume that they are of the form

$$ \textcolor{red}{B}\textcolor{blue}{BBA}\underbrace{\textcolor{blue}{BBBB......BBBB}}_{\text{$2(K-1)$ length}}C_{n-(2K+1)}(L^{\prime}).$$

Since such a chain can satisfy \ref{eq} only if it is actually of the form

$$ \textcolor{red}{B}\textcolor{blue}{BB}\underbrace{\textcolor{blue}{BBBB......BBBB}}_{\text{$2(K-1)$ length}}C_{n-(2K+1)}(L^{\prime})$$

and then we can change its corresponding $g\operatorname{-}C$ by adding $2^{n-2}$ to its initiative value $L$. Therefore, the new $g\operatorname{-}C$ will be in $I_1$, and its corresponding $I_1$ $p\operatorname{-}C$ will be of the form (after assuming the the last cell in $p\operatorname{-}C$ always contains $B$)

\begin{equation*}
	\begin{gathered}
		\textcolor{red}{B}\textcolor{blue}{BBA}\underbrace{\textcolor{blue}{BBBB......BBBB}}_{\text{$2(K-1)$length}}C_{n-(2K+1)}(L^{\prime})
		\\
		or 
		\\
		\textcolor{red}{B}\textcolor{blue}{BABA}\underbrace{\textcolor{blue}{BBBB......BBBB}}_{\text{$2(K-1)$length}}C_{n-(2K+1)}(L^{\prime})
	\end{gathered}
\end{equation*}

which are not satisfying \ref{eq}, and definitely we did not count such a chain in $I_1$ $p\operatorname{-}C$, since we assumed that they are of the form 

$$\textcolor{red}{B}\textcolor{blue}{BB}C_{n-3}(L^{\prime}).$$

Now we are ready to construct our formula. The first step is assuming that integers in $[1,2^n]$ are Collatz numbers (such that $n \geq 7$). After that (and according to the general formula of $I_K$ $p\operatorname{-}C$) we need to find the largest $K$ such that

\begin{equation}\label{eq5}
	\frac{2^{n-1}}{3^{\lceil \frac{K-1}{K} \rceil + n -(2K+1)}} < 1.
\end{equation}

From \ref{eq5} we can deduce that

\begin{equation*}
	\begin{gathered}
		\lceil \frac{K-1}{K} \rceil > (n-1) \log_3(2) +2K +1 -n
		\\
		\Downarrow
		\\
		\frac{K-1}{K}  > \lfloor (n-1) \log_3(2) \rfloor +2K +1 -n
		\\
		\Downarrow
		\\
		0 > 2K^2 + (\lfloor (n-1) \log_3(2) \rfloor -n)K +1.
	\end{gathered}
\end{equation*}

Therefore, our concerned $K$ can take values up to $g(n)$ such that

\begin{equation}\label{eq6}
	g(n \geq 7)=\lfloor \frac{n- \lfloor (n-1) \log_3(2) \rfloor + \sqrt{(n- \lfloor (n-1) \log_3(2) \rfloor)^2 -8}}{4} \rfloor.
\end{equation}

Now, for a given $n$, we need to find the possible number of $I_K$ $p\operatorname{-}C$ such that $1 \leq K \leq g(n)$, i.e. the possible number of combinations of $C_{n-(2K+1)}$. In order to do that, we can use the same trick we used before in \ref{eq3}, which means that we will depend on the free functions in $C_{n-(2K+1)}$. Note that for each $1 \leq K \leq g(n)$ there is a corresponding $p\operatorname{-}C$ which has no free functions, i.e. $p\operatorname{-}C$ which start with

$$C_{n-(2K+1)}=BABABA.....BABABA(L^{\prime}).$$ 

As for the largest possible number of free functions for each $K$, it can be obtained from the following inequality

\begin{equation*}
	\begin{gathered}
		\frac{2^{n-1}}{3^{\lceil \frac{K-1}{K} \rceil + n -(2K+1)-q}} < 1.
		\\
		\Downarrow
		\\
		q < \lceil \frac{K-1}{K} \rceil + n -(2K+1) - (n-1)\log_3(2)
	\end{gathered}
\end{equation*}

Hence, $q$ can take values up to $h(K)$ such that

\begin{equation}\label{eq7}
	h(K)= \lfloor \lceil \frac{K-1}{K} \rceil + n -(2K+1) - (n-1)\log_3(2) \rfloor
\end{equation}

i.e. for each $K \leq g(n)$, there is a function $h(K)$ such that we can construct many $p\operatorname{-}C$ which start with $C_{n-(2K+1)}$ that contains a number of free functions $q \leq h(K)$.

Now, we need to determine the number of allowed movements for each free function, that actually can be obtained from the following formula which is similar (in purpose) to the function $w$. Note that the number of allowed movements for the $r \operatorname{-} th$ function varies according to the total number of free functions (namely $q$).

\begin{equation}
	z(s)= n-(2K+q+1) - \frac{\lceil \frac{\log_3(2)\textcolor{red}{s}-1}{1-\log_3(2)} \rceil + \sqrt{ \left( \lceil \frac{\log_3(2)\textcolor{red}{s}-1}{1-\log_3(2)} \rceil \right)^2}}{2}
\end{equation}

Finally, we ar ready to construct our formula which computes the number of $p\operatorname{-}C$ in the interval $[2^n+1,2^{n+1}]$. 

\begin{equation}
	\delta (n) = g(n) +\sum_{K=1}^{g(n)} \sum_{q=1}^{h(K)} \left( \sum_{r_1=1}^{z(q)} \sum_{r_2=r_1}^{z(q-1)}.....\sum_{r_e=r_{e-1}}^{z(1)} (1) \right).
\end{equation}

For example; at $n=14$ we find that

\begin{equation*}
	\begin{gathered}
		\delta (14) = 2 + \sum_{K=1}^{2} \sum_{q=1}^{h(K)} \left( \sum_{r_1=1}^{z(q)} \sum_{r_2=r_1}^{z(q-1)}.....\sum_{r_e=r_{e-1}}^{z(1)} (1) \right) = 2 + \sum_{q=1}^{2} \left( \sum_{r_1=1}^{z(q)} \sum_{r_2=r_1}^{z(q-1)}.....\sum_{r_e=r_{e-1}}^{z(1)} (1) \right) + 
		\\
		\sum_{q=1}^{1} \left( \sum_{r_1=1}^{z(q)} \sum_{r_2=r_1}^{z(q-1)}.....\sum_{r_e=r_{e-1}}^{z(1)} (1) \right) = 2 + \sum_{r_1=1}^{10}(1) + \sum_{r_1=1}^{8} \sum_{r_2=r_1}^{9}(1) + \sum_{r_1=1}^{8}(1) = 64.
	\end{gathered} 
\end{equation*}

In the following table we give some counts to $\delta (n)$ and $T(n)$, and then we graph them using proportional scales on the $y$ axis. By comparing the two graphs, we find that their curves are almost identical, which means that $\delta (n)$ and $T(n)$ are almost proportional (at least in the range of our computations). Therefore, we think it is reasonable to give the following conjecture.

\begin{conjecture}
	$T(n)$ equals to some constant multiplied by $\delta (n)$ as $n\to \infty$.
\end{conjecture}

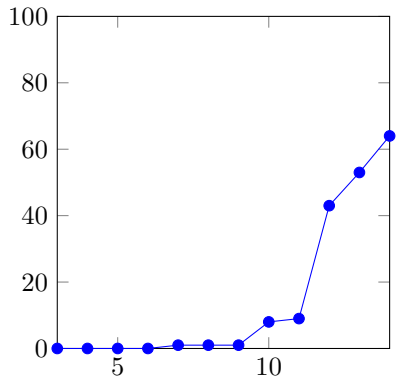
\begin{center}
	\begin{tabular}{|c|c|c|}
		\hline
		$n$ & $\delta (n)$ & $ T(n)$
		\\
		\hline
		3 & 0 & 0
		\\
		\hline
		4 & 0 & 1
		\\
		\hline
		5 & 0 & 3 
		\\
		\hline
		6 & 0 & 0
		\\
		\hline
		7 & 1 & 6
		\\
		\hline
		8 & 1 & 8       
		\\
		\hline
		9 & 1 & 1
		\\
		\hline
		10 & 8 & 31
		\\
		\hline
		11 & 9 & 41
		\\
		\hline
		12 & 43 & 122
		\\
		\hline
		13 & 53 & 255
		\\
		\hline
		14 & 64 & 257
		\\
		\hline
	\end{tabular}
	\qquad
	\begin{tabular}{|c|c|c|}
		\hline
		$n$ & $\delta (n)$ & $ T(n)$
		\\
		\hline
		15 & 261 & 816
		\\
		\hline
		16 & 337 & 1412
		\\
		\hline
		17 & 426 & 1374
		\\
		\hline
		18 & 1580 & 4815
		\\
		\hline
		19 & 2109 & 7310
		\\
		\hline
		20 & 6949 & 18125
		\\
		\hline
		21 & 9705 & 38154
		\\
		\hline
		22 & 13242 & 48062
		\\
		\hline
		23 & 42398 & 123088
		\\
		\hline
		24 & 60109 & 225107
		\\
		\hline
		25 & 83390 & 272512
		\\
		\hline
	\end{tabular}
\end{center}

\begin{center}
	\begin{tikzpicture}
		\begin{axis}[
			height=6cm,
			width=6cm,
			xmin=3, xmax=14,
			ymin=0, ymax=100,
			yticklabel style={
				/pgf/number format/fixed,
			}  ,
			scaled y ticks=false
			]
			
			\addplot[color=blue, mark=*] coordinates {
				(3,0)
				(4,0)
				(5,0)
				(6,0)
				(7,1)
				(8,1)
				(9,1)
				(10,8)
				(11,9)
				(12,43)
				(13,53)
				(14,64)
				
			};	
			
		\end{axis}
		
	\end{tikzpicture}
	\hfil
	
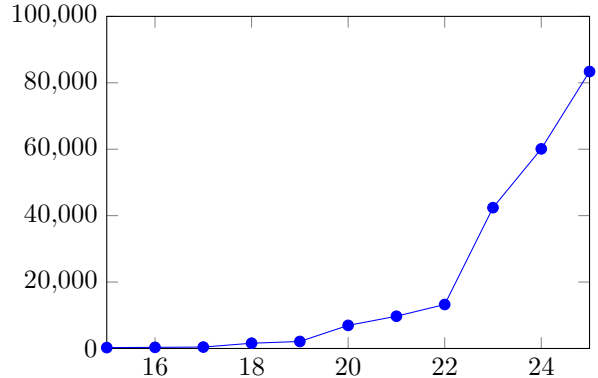
\begin{tikzpicture}
		
		\begin{axis}[
			height=6cm,
			width=8cm,
			xmin=15, xmax=25,
			ymin=0, ymax=100000,
			yticklabel style={
				/pgf/number format/fixed,
			}  ,
			scaled y ticks=false
			]
			
			\addplot[color=blue, mark=*] coordinates {
				(15,261)
				(16,337)
				(17,426)
				(18,1580)
				(19,2109)
				(20,6949)
				(21,9705)
				(22,13242)
				(23,42398)
				(24,60109)
				(25,83390)
				
			};	
			
		\end{axis}
		
	\end{tikzpicture}
\end{center}

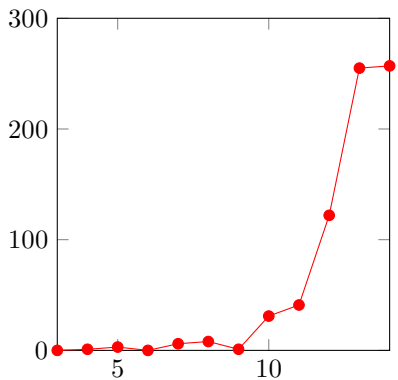
\begin{center}
	\begin{tikzpicture}
		
		\begin{axis}[
			height=6cm,
			width=6cm,
			xmin=3, xmax=14,
			ymin=0, ymax=300,
			yticklabel style={
				/pgf/number format/fixed,
			}  ,
			scaled y ticks=false
			]
			
			\addplot[color=red, mark=*] coordinates {
				(3,0)
				(4,1)
				(5,3)
				(6,0)
				(7,6)
				(8,8)
				(9,1)
				(10,31)
				(11,41)
				(12,122)
				(13,255)
				(14,257)
			};	
			
		\end{axis}
		
	\end{tikzpicture}
	\hfil
	
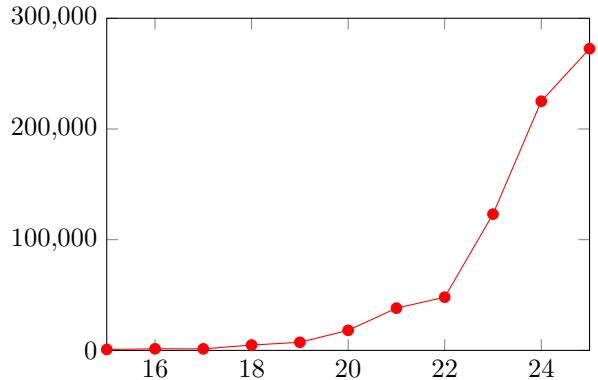
\begin{tikzpicture}
		
		\begin{axis}[
			height=6cm,
			width=8cm,
			xmin=15, xmax=25,
			ymin=0, ymax=300000,
			yticklabel style={
				/pgf/number format/fixed,
			}  ,
			scaled y ticks=false
			]
			
			\addplot[color=red, mark=*] coordinates {
				(15,816)
				(16,1412)
				(17,1374)
				(18,4815)
				(19,7310)
				(20,18125)
				(21,38154)
				(22,48062)
				(23,123088)
				(24,225107)
				(25,272512)
				
			};	
			
		\end{axis}
		
	\end{tikzpicture}
	
\end{center}

\section*{Acknowledgements} I am very grateful to Dr. Mohamed Anwar, Assistant professor at Mathematics department, Faculty of Science, Ain Shams University, Egypt for all the help and guidance that he has given me.


\begin{thebibliography}{99}
	
	\bibitem{ars} Arslan, Abdullah N. (2018). Methods for constructing Collatz numbers. Notes on Number Theory and Discrete Mathematics, \textbf{24(2)}, 47-54, doi: 10.7546/nntdm.2018.24.2.47-54.

\bibitem{Kor} I. Korec, "A density estimate for the 3x+1 problem", Math. Slovaca 44 (1994), no.1, 85-89.

\bibitem{kra} I. Krasikov, J. Lagarias, Bounds for the 3x + 1 problem using difference inequalities, Acta Arith. \textbf{109} (2003), 237–258.

\bibitem{ALL} J.-P. ALLOUCHE, "Sur la conjecture de 'Syracuse-Kakutani-Collatz'" Seminaire de Théorie des Nombres de Bordeaux (1978-1979), Volume: 8, page 1-16.

\bibitem{Ter} R. Terras, "A stopping time problem on the positive integers", Acta Arith. 30 (1976), 241--252.

\bibitem{Tao} Tao, T. Almost all orbits of the Collatz map attain almost bounded values. arXiv 2019, arXiv:1909.03562v3

\bibitem{zar} Zarnowski, R. (2020). A refinement of the 3x + 1 conjecture. Notes on Number Theory and Discrete Mathematics, \textbf{26(3)}, 234-244, doi: 10.7546/nntdm.2020.26.3.234-244.

\end{thebibliography}
\end{document}